\documentclass[11pt]{amsart}

\usepackage[margin=1.6in]{geometry}
\usepackage{amsthm} 
\usepackage{amssymb}
\usepackage{amsfonts}
\usepackage{latexsym}
\usepackage{mathrsfs}
\usepackage{bm}
\usepackage{yhmath}
\usepackage{tikz}
\usepackage{etoolbox}

\theoremstyle{plain}
\newtheorem{theorem}{Theorem}[section]
\newtheorem{lemma}[theorem]{Lemma}
\newtheorem{proposition}[theorem]{Proposition}
\newtheorem{corollary}[theorem]{Corollary}

\newtheorem*{BSEtheorem}{Theorem \ref{thm:BSE theorem}}
\newtheorem*{maintheorem}{Theorem \ref{thm:main theorem}}

\newtheorem{claim}{Claim}

\theoremstyle{definition}
\newtheorem{definition}[theorem]{Definition}

\theoremstyle{remark}
\newtheorem{remark}{Remark}

\numberwithin{equation}{section}

\allowdisplaybreaks

\makeatletter
\def\@citestyle{\m@th\upshape\mdseries}
\def\citeform#1{{\bfseries#1}}
\def\@cite#1#2{{%
  \@citestyle[\citeform{#1}\if@tempswa, #2\fi]}}
\@ifundefined{cite }{%
  \expandafter\let\csname cite \endcsname\cite
  \edef\cite{\@nx\protect\@xp\@nx\csname cite \endcsname}%
}{}
\makeatother

\renewcommand{\leq}{\leqslant}
\renewcommand{\geq}{\geqslant}

\newcommand{\inner}[2]{\langle #1\,,#2\rangle}
\newcommand{\biginner}[2]{\left\langle #1\,,\,#2\right\rangle}

\newcommand{\ddl}[2]{\frac{d{#1}}{d{#2}}}

\newcommand{\ppl}[2]{\frac{\partial{#1}}{\partial{#2}}}
\newcommand{\ppz}[2]{\frac{\partial^2{#1}}{\partial{#2}^2}}

\newcommand{\C}{\mathbb{C}}
\newcommand{\R}{\mathbb{R}}

\newcommand{\B}{\mathbb{B}}
\renewcommand{\H}{\mathbb{H}}

\newcommand{\T}{\mathbf{T}}

\DeclareMathOperator{\dist}{dist}%

\DeclareMathOperator{\Area}{Area}%
\DeclareMathOperator{\Length}{Length}%
\DeclareMathOperator{\Mob}{M\text{\"o}b}%
\DeclareMathOperator{\Ric}{Ric}%

\newcommand{\CC}{\mathcal{C}}
\newcommand{\Lcal}{\mathcal{L}}

\newcommand{\Nscr}{\mathscr{N}}

\begin{document}

\title{Least Area Spherical Catenoids in Hyperbolic
Three-Dimensional Space}
%\subtitle{Least area Spherical Catenoids in $\H^3$}

\author{Biao Wang}
%\date{}
\date{\today}

\subjclass{Primary 53A10, Secondary 53C42}
% 57M05 Low-dimensional topology
% 53A10 Minimal surfaces, surfaces with prescribed mean curvature
% 53C42 Immersions (minimal, prescribed curvature, tight, etc.)
% 53C44 Geometric evolution equations (mean curvature flow)
\address{Department of Mathematics and Computer Science\\
           QCC, The City University of New York\\
           222-05 56th Avenue Bayside, NY 11364\\}
\email{biwang@qcc.cuny.edu}

\maketitle

\begin{abstract}
  For a family of spherical minimal catenoids $\{\CC_a\}_{a>0}$
  in the $3$-dimensional hyperbolic space $\H^3$ (see
  $\S$\ref{subsec:catnoids} for detail definitions),
  there exists two constants $0<a_c<a_l$ such that
  the following are true:
  \begin{itemize}
     \item $\CC_a$ is an unstable minimal surface with index
           one if $a<a_c$,
     \item $\CC_a$ is a stable minimal surface if $a\geq{}a_c$, and
     \item $\CC_a$ is a least area minimal surface in the sense of
           Meeks-Yau (see $\S$\ref{subsec:meek-yau least area} for the
           definition) if $a\geq{}a_l$.
  \end{itemize}
\end{abstract}

%\tableofcontents%

%====================================================================
\section{Introduction}\label{sec:Introduction}

Suppose that $\Sigma$ is a surface immersed
in a $3$-dimensional Riemannian manifold
$M$. We pick up a local orthonormal frame field
$\{e_1,e_2,e_3\}$ for $M$ such that, restricted to $\Sigma$,
the vectors $\{e_1,e_2\}$ are tangent to $\Sigma$ and the
vector $e_3$ is perpendicular to $\Sigma$. Let
$A=(h_{ij})_{2\times{}2}$ be the second fundamental
form of $\Sigma$, whose entries $h_{ij}$ are represented by
\begin{equation*}
   h_{ij}=\inner{\nabla_{e_i}e_{3}}{e_{j}}\ ,
   \quad{}i,j=1,2\ ,
\end{equation*}
where $\nabla$ is the covariant derivative in $M$, and
$\inner{\cdot}{\cdot}$ is the metric of $M$.

\subsection{Basic minimal surfaces}
An immersed surface $\Sigma\subset{}M$ is called a
\emph{minimal surface} if its \emph{mean curvature}
$H=h_{11}+h_{22}$ is identically zero.
For any immersed minimal surface $\Sigma$
in $M$, the {\em Jacobi operator} on $\Sigma$ is
\begin{equation}\label{eq:Jacobi operator}
   \Lcal=\Delta_{\Sigma}+(|A|^2+\Ric(e_3))\ ,
\end{equation}
where $\Delta_{\Sigma}$ is the Lapalican on $\Sigma$,
$|A|^2=\sum_{i,j=1}^{2}h_{ij}^2$ is
the length of the second fundamental form on $\Sigma$ and
$\Ric(e_3)$ is the Ricci curvature of $M$
in the direction $e_3$.

Suppose that $\Sigma$ is a complete minimal surface immersed in
a complete Riemannian $3$-manifold $M$.
For any compact connected subdomain $\Omega$ of $\Sigma$,
its first eigenvalue is defined by
\begin{equation}\label{eq:1st eigenvalue of Omega}
   \lambda_{1}(\Omega)=\inf\left\{-\int_{\Omega}f\Lcal{}f
   \ \left|\ f\in{}C_{0}^\infty(\Omega)\ \text{and}\
   \int_{\Omega}f^2=1\right.\right\}\ .
\end{equation}
We say that $\Omega$ is \emph{stable} if $\lambda_{1}(\Omega)>0$,
\emph{unstable} if $\lambda_{1}(\Omega)<0$ and
\emph{maximally weakly stable} if $\lambda_{1}(\Omega)=0$.

\begin{lemma}\label{lem:monotonicity of eigenvalue}
Suppose that $\Omega_1$ and $\Omega_2$ are connected
subdomains of $\Sigma$ with $\Omega_1\subset\Omega_2$, then
\begin{equation*}
  \lambda_{1}(\Omega_1)\geq\lambda_{1}(\Omega_2)\ .
\end{equation*}
If $\Omega_2\setminus\overline{\Omega}_1\ne\emptyset$, then
\begin{equation*}
  \lambda_{1}(\Omega_1)>\lambda_{1}(\Omega_2)\ .
\end{equation*}
\end{lemma}

\begin{remark}If $\Omega\subset\Sigma$ is maximally weakly
stable, then for any compact connected subdomains
$\Omega_1,\Omega_2\subset\Sigma$ satisfying
$\Omega_1\subsetneq\Omega\subsetneq\Omega_2$, we have that
$\Omega_1$ is stable whereas $\Omega_2$ is unstable.
\end{remark}

Let $\Omega_1\subset\Omega_2\subset\cdots\subset\Omega_n\subset\cdots$
be an exhaustion of $\Sigma$, then the first eigenvalue of
$\Sigma$ is defined by
\begin{equation}\label{eq:1st eigenvalue of Sigma}
   \lambda_{1}(\Sigma)=\lim_{n\to\infty}\lambda_{1}(\Omega_n)\ .
\end{equation}
This definition is independent of the choice of the exhaustion.
We say that $\Sigma$ is \emph{globally stable} or
\emph{stable} if $\lambda_{1}(\Sigma)>0$ and
\emph{unstable} if $\lambda_{1}(\Sigma)<0$.

The following theorem was proved by Fischer-Colbrie and Schoen in
\cite[Theorem 1]{FCS80} (see also \cite[Proposition 1.39]{CM11}).

\begin{theorem}[Fischer-Colbrie and Schoen]\label{thm:FCS80}
Let $\Sigma$ be a complete minimal surface in a complete
Riemannian $3$-manifold $M$, then $\Sigma$ is stable
if and only if there exists a positive function
$\phi:\Sigma\to\R$ such that $\Lcal{}\phi=0$.
\end{theorem}

The \emph{Morse index} of compact connected subdomain
$\Omega$ of $\Sigma$ is the number of negative eigenvalues of the
Jacobi operator $\Lcal$ (counting with multiplicity) acting on the
space of smooth sections of the normal bundle that vanishes on
$\partial\Omega$. The \emph{Morse index} of $\Sigma$ is the
supremum of the Morse indices of compact subdomains of $\Sigma$.

\subsection{Least area minimal annuli in the sense of Meeks-Yau
(\cite[p.~412]{MY1982(t)})}\label{subsec:meek-yau least area}

Suppose that $\Sigma$ is a complete minimal surface immersed in
a complete Riemannian $3$-manifold $M$. For any compact subdomain
$\Omega$ of $\Sigma$, it is said to be \emph{least area} if its
area is smaller than that of any other surface in the same
homotopic class with the same boundary as $\partial\Omega$. We
say that $\Sigma$ is a \emph{least area minimal surface} if
any compact subdomain of $\Sigma$ is least area.

Let $S$ be a compact annulus-type minimal surface immersed in a Riemannian
$3$-manifold $M$. Suppose that the boundary of $S$ is the union of two simple
closed curves $C_1,C_2$ which bound two least area minimal disks $D_1,D_2$
respectively. The annulus $S$ is called a
\emph{least area minimal surface in the sense of Meeks-Yau} in $M$ if
$S$ is a least area minimal annulus in the regular sense and
\begin{equation*}
   \Area(S)<\Area(D_1)+\Area(D_2)\ ,
\end{equation*}
where $\Area(\cdot)$ denotes the area of the surface in $M$.
A complete annulus-type minimal surface $\Sigma$ immersed in $M$ is called a
\emph{least area minimal surface in the sense of Meeks-Yau} if any annulus-type
compact subdomain of $\Sigma$, which is homotopically equivalent to $\Sigma$,
is a least area minimal surface in the sense of Meeks-Yau.

\subsection{Main statements}
Do Carmo and Dajczer studied three types of rotationally symmetric
minimal hypersurfaces in $\H^{n+1}$ in \cite{dCD83}.
A rotationally symmetric minimal hypersurface is called a
{\em spherical} catenoid if it is foliated by spheres, a
{\em hyperbolic} catenoid if it is foliated by totally geodesic
hyperplanes, and a {\em parabolic} catenoid if it is foliated
by horospheres. Do Carmo and Dajczer
proved that the hyperbolic and parabolic catenoids are
globally stable (see \cite[Theorem 5.5]{dCD83}), then Candel proved
that the hyperbolic and parabolic catenoids are least area
minimal surfaces (see \cite[p. 3574]{Can07}).

In this paper, we will study the spherical catenoids in $\H^3$.
Compared with the hyperbolic and parabolic catenoids,
the spherical catenoids are more complicated. Let $\B^2_{+}$ be the
upper half unit disk on the $xy$-plane with the warped product
metric given by \eqref{eq:warped product metric}, and let
$\sigma_a$ be the
catenary given by \eqref{eq:parametric equation of catenary},
which is symmetric
about the $y$-axis and passes through the point $(0,a)\in\B_{+}^2$,
where $a>0$ is the hyperbolic distance from the catenary $\sigma_a$ to
the origin. Let $\CC_a$ be the spherical minimal catenoid
generated by $\sigma_a$. Mori, Do Carmo and Dajczer,
B{\'e}rard and Sa Earp, and Seo proved the following result.

\begin{theorem}[\cite{BSE10,dCD83,Mor81,Seo11}]
\label{thm:stability of catenoids}
There exist two constants $A_1\approx{}0.46288$ and
$A_{2}=\frac{1}{2}\cosh^{-1}
\left(\sqrt{\frac{11+8\sqrt{2}}{7}}\right)\approx{}0.5915$
such that $\CC_a$ is unstable if $0<a<A_1$, and $\CC_a$
is globally stable if $a>A_2$.
\end{theorem}

\begin{remark}The constants $A_1$ and $A_2$ were given by Seo
in \cite[Corollary 4.2]{Seo11} and by B{\'e}rard and Sa Earp
in \cite[Lemma 4.4]{BSE10} respectively.
A few years ago, Do Carmo and Dajczer showed that $\CC_a$
is unstable if $a\lessapprox{}0.42315$ in \cite{dCD83},
and Mori showed that $\CC_a$ is stable if
$a>\cosh^{-1}(3)\approx{}1.7627$
in \cite{Mor81} (see also \cite[p. 34]{BSE09}).
\end{remark}

According to the numerical computation, B{\'e}rard and Sa Earp
claimed that $A_1=A_2$. More precisely, we have the
following theorem.

\begin{BSEtheorem}
There exists a constant $a_c\approx{}0.49577389$ such that
the following statements are true:
\begin{enumerate}
     \item $\CC_a$ is an unstable minimal surface with index
           one if $0<a<a_c$;
     \item $\CC_a$ is a globally stable minimal surface if
           $a\geq{}a_c$.
  \end{enumerate}
\end{BSEtheorem}

%\begin{theorem}[\cite{BSE10}]\label{thm:BSE theorem}
%There exists a constant $a_c\approx{}0.49577389$ such that
%the following statements are true:
%\begin{enumerate}
%     \item $\CC_a$ is an unstable minimal surface with index
%           one if $0<a<a_c$;
%     \item $\CC_a$ is a globally stable minimal surface if
%           $a\geq{}a_c$.
%\end{theorem}

Similar to the case of hyperbolic and parabolic catenoids,
we want to know whether the globally stable spherical catenoids
are least area minimal surfaces.
In this paper, we prove that there exists a positive number
$a_l$ such that $\CC_a$ is a least area minimal surface
if $a\geq{}a_l$. More precisely, we will prove the following
result.

\begin{maintheorem}
There exists a constant $a_l\approx{}1.10055$ defined by
\eqref{eq:lambda_c} such that for any
$a\geq{}a_l$ the catenoid $\CC_a$ is a least area
minimal surface in the sense of Meeks-Yau.
\end{maintheorem}

%\begin{theorem}\label{thm:main theorem}
%There exists a constant $a_l\approx{}1.10055$ such that for any
%$a\geq{}a_l$ the catenoid $\CC_a$ is a least area
%minimal surface in the sense of Meeks-Yau.
%\end{theorem}

\subsection{Plan of the paper}
This paper is organized as follows.
In $\S~$\ref{sec:prelim} we introduce the minimal
spherical catenoids in $\H^3$.
In $\S$~\ref{sec:proof of BSE theorem} we introduce
Jacobi fields on the catenoids
(following B{\'e}rard and Sa Earp in \cite{BSE10}) and
prove Theorem~\ref{thm:BSE theorem}.
In $\S$~\ref{sec:proof of main theorem} we prove
Theorem~\ref{thm:main theorem}.

%===================================================================
\section{Preliminaries}\label{sec:prelim}

In this paper, we work in the Pinecar{\'e} ball model of
$\B^{3}$, i.e.,
\begin{equation*}
   \B^{3}=\{(u,v,w)\in\R^{3}\ | \ u^{2}+v^{2}+w^{2}<{}1\},
\end{equation*}
equipped with the hyperbolic metric
\begin{equation*}
   ds^{2}=\frac{4(du^{2}+dv^{2}+dw^{2})}{(1-r^{2})^{2}}\ ,
\end{equation*}
where $r=\sqrt{u^{2}+v^{2}+w^{2}}$. The hyperbolic space
$\B^{3}$ has a natural compactification:
$\overline{\B^{3}}=\B^{3}\cup{}S_{\infty}^{2}$,
where $S_{\infty}^{2}\cong\C\cup\{\infty\}$ is called the
\emph{Riemann sphere}. The orientation preserving
isometry group of $\B^3$ is denoted by $\Mob(\B^3)$, which
consists of M\"obius transformations that preserve the unit
ball $\B^{3}$ (see \cite[Theorem 1.7]{MT98}).

Let $X$ be a subset of $\B^{3}$, we define
the {\em asymptotic boundary} of $X$ by
\begin{equation}
   \partial_{\infty}X=\overline{X}
   \cap{}S_{\infty}^{2}\ ,
\end{equation}
where $\overline{X}$ is the closure of $X$ in
$\overline{\B}{}^{3}$.

Using the above notation, we have
$\partial_\infty\B^3=S_{\infty}^2$. If $P$ is a geodesic plane
in $\B^{3}$, then $P$ is perpendicular to $S_{\infty}^{2}$ and
$C\stackrel{\text{def}}{=}\partial_{\infty}P$ is
an Euclidean circle on $S_{\infty}^{2}$. We also say that
$P$ is {\em asymptotic to} $C$.

%====================================================================
\begin{figure}[htbp]
  \begin{center}
     \includegraphics[scale=0.9]{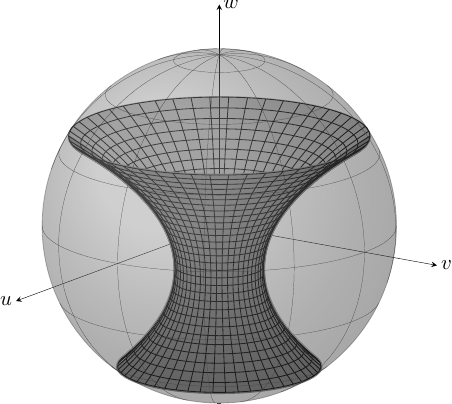}
   \end{center}
  \caption{A surface of revolution in the hyperbolic $3$-space $\B^3$
           whose rotation axis is the $w$-axis.}\label{fig:catenoid in B3}
\end{figure}
%====================================================================

Suppose that $G$ is a subgroup of $\Mob(\B^{3})$ that leaves
a geodesic $\gamma\subset\B^{3}$ pointwise fixed. We call $G$ the
\emph{spherical group} of $\B^{3}$ and $\gamma$ the
\emph{rotation axis} of $G$. A surface in $\B^{3}$ invariant
under $G$ is called a \emph{spherical surface} or a
\emph{surface of revolution}
(see Fig.~\ref{fig:catenoid in B3}). For two circles $C_{1}$
and $C_{2}$ in $\B^{3}$, if there is a geodesic $\gamma$,
such that each of $C_{1}$ and $C_{2}$ is invariant under the
group of rotations that fixes $\gamma$ pointwise, then $C_{1}$
and $C_{2}$ are said to be \emph{coaxial}, and $\gamma$ is
called the {\em rotation axis} of $C_{1}$ and $C_{2}$.

%-------------------------------------------------------------------
\subsection{Minimal spherical catenoids in $\B^3$}
\label{subsec:catnoids}

In this subsection, we follow Hsiang (see \cite{BCH09,Hsi82}) to
introduce the minimal spherical catenoids in $\B^3$.

Suppose that $G$ is the spherical group of $\B^{3}$ along
the geodesic
\begin{equation}\label{eq:rotation axis}
   \gamma_{0}=\{(u,0,0)\in\B^3\ |\ -1<u<1\}\ ,
\end{equation}
then $\B^3/G\cong\B_{+}^2$, where
\begin{equation}
   \B_{+}^2=\{(u,v)\in\B^2\ |\ v\geq{}0\}\ .
\end{equation}

For any point $p=(u,v)\in\B_{+}^2$, there is a unique
geodesic segment $\gamma'$ passing through $p$ that is
perpendicular to $\gamma_{0}$ at $q$.
Let $x=\dist(O,q)$ and $y=\dist(p,q)=\dist(p,\gamma_{0})$
(see Fig.~\ref{fig:intrinsic metric}), where
$\dist(\cdot,\cdot)$ denotes the hyperbolic distance, then by
\cite[Theorem 7.11.2]{Bea95}, we have
\begin{equation}\label{eq:(x,y) in terms of (u,v)}
   \tanh{}x=\frac{2u}{1+(u^2+v^2)}
   \quad\text{and}\quad
   \sinh{}y=\frac{2v}{1-(u^2+v^2)}\ .
\end{equation}
Equivalently, we also have
\begin{equation}\label{eq:(u,v) in terms of (x,y)}
   u=\frac{\sinh{}x\cosh{}y}{1+\cosh{}x\cosh{}y}
   \quad\text{and}\quad
   v=\frac{\sinh{}y}{1+\cosh{}x\cosh{}y}\ .
\end{equation}

%====================================================================
\begin{figure}[htbp]
  \begin{center}
     \includegraphics[scale=1]{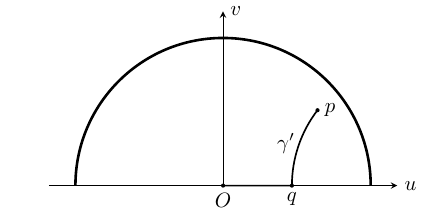}
   \end{center}
  \caption{For a point $p$ in $\B_{+}^2$ with the warped product
   metric, its coordinates $(x,y)$ are defined by $x=\dist(O,q)$
   and $y=\dist(p,q)$.}\label{fig:intrinsic metric}
\end{figure}
%====================================================================

It's well known that $\B_{+}^2$ can be equipped with
the \emph{metric of warped product} in terms
of the parameters $x$ and $y$ as follows:
\begin{equation}\label{eq:warped product metric}
   ds^2=\cosh^{2}y\cdot{}dx^2+dy^2\ ,
\end{equation}
where $dx$ represents the hyperbolic metric on the geodesic
$\gamma_0$ in \eqref{eq:rotation axis}. We call the horizontal
geodesic $\{(u,0)\in\B_{+}^2\ |\ -1<u<1\}$ the $x$-axis and
the vertical geodesic $\{(0,v)\in\B_{+}^2\ |\ 0\leq{}v<1\}$
the $y$-axis. The orientations of the $x$-axis and the $y$-axis
are considered to be the same as that of the $u$-axis and the
$v$-axis respectively. Thus we also consider that the $x$-axis
and the $y$-axis are equivalent to the $u$-axis and the $v$-axis
respectively.

If $\CC$ is a minimal surface of revolution in $\B^3$ with
respect to the axis $\gamma_{0}$, then the curve
$\sigma=\CC\cap\B_{+}^{2}$ is called the
\emph{generating curve} of $\CC$.
Suppose that $\sigma$ is given by the parametric equations:
$x=x(s)$ and $y=y(s)$, where $s\in(-\infty,\infty)$ is an
arc length parameter of $\sigma$. By the argument in
\cite[pp. 486--488]{Hsi82}, the curve $\sigma$ satisfies the
following equations
\begin{equation}\label{eq:differential equations of Pi}
   \frac{2\pi\sinh{}y\cdot\cosh^{2}y}
   {\sqrt{\cosh^2{}y+(y')^2}}=
   2\pi\sinh{}y\cdot\cosh{}y\cdot\sin\theta=k\
   (\text{constant})\ ,
\end{equation}
where $y'=dy/dx$ and $\theta$ is the angle between the tangent
vector of $\sigma$ and the vector $e_y=\partial/\partial{}y$
at the point $(x(s),y(s))$ (see Fig.~\ref{fig:diff eq of sigma}).

%===================================================================
\begin{figure}[htbp]
  \begin{center}
     \includegraphics[scale=0.9]{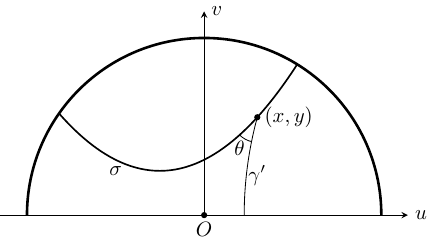}
   \end{center}
  \caption{$\theta$ is the angle between the parametrized
  curve $\sigma$ and the geodesic $\gamma'$ at the point
  $(x(s),y(s))\in\sigma\cap\gamma'$, where $\gamma'$ is
  perpendicular to the $u$-axis.}\label{fig:diff eq of sigma}
\end{figure}
%===================================================================

By the argument in \cite[pp.54--58]{Gom87}), without loss of
generality, we assume that
the curve $\sigma$ is only symmetric about the $y$-axis
and intersects the $y$-axis orthogonally at $y_0=y(0)$,
and so $y'(0)=0$.
Substitute these to \eqref{eq:differential equations of Pi},
we get $k=2\pi\sinh(y_0)\cosh(y_0)$, and then we get
\begin{equation}\label{eq:angle-alpha}
   \sin\theta=\frac{\sinh(y_0)\cosh(y_0)}{\sinh(y)\cosh(y)}
             =\frac{\sinh(2y_0)}{\sinh(2y)}\ .
\end{equation}
Now solve $x$ in terms of $y$
from \eqref{eq:differential equations of Pi} and take the definite
integral from $y_0$ to $y$ for any $y\geq{}y_0$, we have
\begin{equation}\label{eq: catenary equation}
   x(y)=\int_{y_0}^{y}\frac{\sinh(2y_0)}{\cosh{}y}
             \frac{dy}{\sqrt{\sinh^2(2y)-\sinh^2(2y_0)}}\ .
\end{equation}
Let $y\to\infty$, we get (see Fig.~\ref{fig:Gomes distance})
\begin{equation}\label{eq:Gomes distance}
   x(\infty)=\int_{y_0}^{\infty}\frac{\sinh(2y_0)}{\cosh{}y}
             \frac{dy}{\sqrt{\sinh^2(2y)-\sinh^2(2y_0)}}\ .
\end{equation}

%===================================================================
\begin{figure}[htbp]
  \begin{center}
     \includegraphics[scale=0.9]{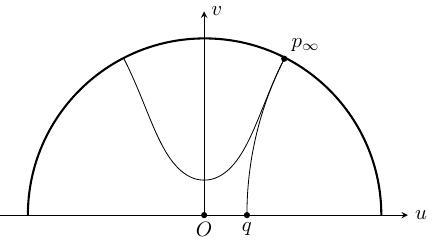}
  \end{center}
  \caption{The distance $x(\infty)$ defined in
  \eqref{eq:Gomes distance} is equal to $\dist(O,q)$, where
  the point $q$ is the intersection of the $u$-axis and the
  unique geodesic which is perpendicular to both the $u$-axis
  at $q$ and $\partial_{\infty}\B_{+}^2$ at $p_\infty$ (here
  $p_\infty$ is one of the asymptotic boundary points of
  $\sigma$ given by \eqref{eq:parametric equation of catenary}).
  In this figure, $y_{0}=0.4$, and so
  $x(\infty)\approx{}0.492681$.}
  \label{fig:Gomes distance}
\end{figure}
%====================================================================

Now replace $y_0$ by a parameter $a\in[0,\infty)$ in
\eqref{eq: catenary equation}, and set
\begin{equation}\tag{\ref{eq: catenary equation}$'$}
                \label{eq:catenary(a,t)}
  \rho(a,t)=\int_{a}^{t}\frac{\sinh(2a)}{\cosh\tau}
  \frac{d\tau}{\sqrt{\sinh^2(2\tau)-\sinh^2(2a)}}\ ,
  \quad{}t\geq{}a\ .
\end{equation}
Let $\sigma_a$ be the catenary whose parametric equation
is given by
\begin{equation}\label{eq:parametric equation of catenary}
   t\mapsto(\pm\rho(a,t),t)\in\B_{+}^2\ ,
   \quad\text{for}\ t\geq{}a\ .
\end{equation}
Let $\CC_a$ be the minimal surface
of revolution along the axis $\gamma_{0}$ whose generating
curve is the catenary $\sigma_a$.

%-----------------------------------------------------------
\subsection{Existence and uniqueness of spherical catenoids}
Obviously the asymptotic boundary of any spherical catenoid
$\CC_a$ is the union of two circles (see also
\cite[Proposition 3.1]{Gom87}). It's important for us
to determine whether there exists a minimal spherical
catenoid asymptotic to any given pair of disjoint circles on
$S_{\infty}^2$, since in \cite{Wang11a} we
construct quasi-Fuchsian $3$-manifolds
which contain arbitrarily many incompressible minimal surface
by using the (least area) minimal spherical catenoids as the
barrier surfaces.

If $C_{1}$ and $C_{2}$ are two disjoint circles on
$S_{\infty}^{2}$, then they are always coaxial. In fact,
let $P_{1}$ and $P_{2}$ be the geodesic planes asymptotic
to $C_{1}$ and $C_{2}$ respectively, there always exists
a unique geodesic $\gamma$ such that $\gamma$ is perpendicular
to both $P_{1}$ and $P_{2}$. Therefore $C_{1}$ and $C_{2}$ are
coaxial with respect to $\gamma$.
We may define the distance between $C_{1}$ and $C_{2}$ by
\begin{equation}
   d_{L}(C_1,C_2)=\dist(P_1,P_2)\ .
\end{equation}

\begin{theorem}\label{thm:Gomes1987-prop3.2-a}
There exists a constant $D_{c}\approx{}1.00228589640$
such that for two disjoint circles
$C_{1},C_{2}\subset{}S_{\infty}^{2}$,
if $d_{L}(C_{1},C_{2})\leq{}D_{c}$, then
there exist a spherical minimal catenoid $\CC$ which is
asymptotic to $C_{1}\cup{}C_{2}$.
\end{theorem}

\begin{remark}In \cite[p. 402]{dOS98}, de Oliveria and Soret
show that for any two congruent circles (in the asymptotic boundary
of the upper half space model of the hyperbolic space $\H^3$)
of Euclidean diameter $d$ and disjoint from each other by the
Euclidean distance $D$, there exists \emph{two} catenoids bounding the
two circles if and only if $D/d\leq\delta$ for some $\delta>0$.
Direct computation shows that
$\delta=\cosh(d_{0}(a_c))-1\approx{}0.127626$, where $d_0(a)$ is
the function defined by \eqref{eq:Gomes function I} and $a_c$ is the
unique critical number of the function $d_0(a)$.
\end{remark}

%===================================================================
\begin{figure}[htbp]
  \begin{center}
     \includegraphics[scale=1]{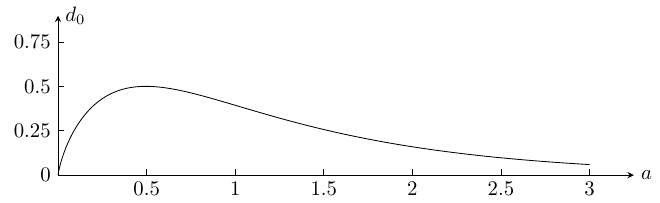}
  \end{center}
  \caption{The graph of the function $d_{0}(a)$ defined by
  \eqref{eq:Gomes function I} for
  $a\in[0,3]$. It seems that $d_{0}(a)$ only has a unique critical number.}
  \label{fig:Gomes function}
\end{figure}
%===================================================================

\begin{proof}[\bf{Proof of Theorem~\ref{thm:Gomes1987-prop3.2-a}}]
At first, we define the following
definite integral (depending on the parameter $a$)
\begin{equation}\label{eq:Gomes function I}
   d_{0}(a)=\int_{a}^{\infty}
             \frac{\sinh(2a)}{\cosh{}t}
             \frac{dt}{\sqrt{\sinh^2(2t)-\sinh^2(2a)}}\ .
\end{equation}

We claim that $d_{0}(0)=0$, and as $a$ increases $d_{0}(a)$
increases monotonically, reaches a maximum, then decreases
asymptotically to zero as $a$ goes to infinity
(see also \cite[Proposition 3.2]{Gom87} and
Fig.~\ref{fig:Gomes function}).

It's easy to show $d_{0}(a)\to{}0$ as $a\to\infty$.
In fact, using the substitution $t\to{}t+a$, we have
\begin{align*}
   d_{0}(a)
      &=\int_{0}^{\infty}\frac{\sinh(2a)}{\cosh(a+t)}
        \frac{dt}{\sqrt{\sinh^2(2a+2t)-\sinh^2(2a)}}\\
      &=\int_{0}^{\infty}\frac{1}{\cosh(t+a)}
        \frac{dt}{\sqrt{\left(\dfrac{\sinh(2a+2t)}
        {\sinh(2a)}\right)^2-1}}\\
      &<\int_{0}^{\infty}\frac{1}{\cosh{}a}
        \frac{dt}{\sqrt{(\sinh(2t)+\cosh(2t))^2-1}}\ .
\end{align*}
Since $\sinh(2t)+\cosh(2t)=e^{2t}$, we have
\begin{equation}\label{eq:estimate on d0}
\begin{aligned}
   d_{0}(a)
       &<\frac{1}{\cosh{}a}\int_{0}^{\infty}
         \frac{dt}{\sqrt{e^{4t}-1}}
        =\frac{1}{\cosh{}a}\int_{0}^{\infty}
         \frac{e^{-2t}}{\sqrt{1-e^{-4t}}}\,dt\\
       &=\frac{1}{\cosh{}a}\cdot\frac{\pi}{4}
         \longrightarrow{}0\quad\text{as}\ a\to\infty\ .
\end{aligned}
\end{equation}
Besides, since $\displaystyle\lim_{a\to{}0^{+}}d_0(a)=0$,
$d_0(a)>0$ for $a\in(0,\infty)$
and $d_{0}(a)\to{}0$ as $a\to\infty$,
it must have at least one  maximum value in $(0,\infty)$.

By the argument in the proof of Theorem~\ref{thm:BSE theorem}
in $\S$~\ref{sec:proof of BSE theorem},
we know that $d_{0}'(a)$ has a unique zero $a_c$ such that
$d_{0}'(a)>0$ if $0<a<a_c$ and $d_{0}'(a)<0$ if $a>a_c$, hence
the proof of the claim is complete.

According to the numerical computation: the function
$d_{0}(a)$ achieves its (unique) maximum value
$\approx{}0.5011429482$ when $a=a_c\approx{}0.49577389$,
and so $D_{c}=2d_{0}(a_c)\approx{}1.0022858964$.
\end{proof}

Theorem \ref{thm:Gomes1987-prop3.2-a} shows the existence
of spherical minimal catenoids.
On the other hand, we also have the uniqueness of catenoids
in the sense of following theorem proved by Levitt and
Rosenberg (see \cite[Theorem 3.2]{LR85} and
\cite[Theorem 3]{dCGT86}).
Recall that a complete minimal surface $\Sigma$ of $\H^3$ is
\emph{regular at infinity} if $\partial_\infty\Sigma$
is a $C^2$-submanifold of $S^2_{\infty}$ and
$\overline\Sigma=\Sigma\cup\partial_\infty\Sigma$ is a
$C^2$-surface (with boundary) of $\overline{\H^3}$.

\begin{theorem}[Levitt and Rosenberg]
Let $C_1$ and $C_2$ be two disjoint round circles on
$S_{\infty}^{2}$ and let $\CC$ be a connected minimal surface
immersed in $\H^3$ with $\partial_{\infty}\CC=C_1\cup{}C_2$
and $\CC$ regular at infinity. Then $\CC$ is a
spherical catenoid.
\end{theorem}

%======================================================================
%======================================================================
\section{Stability of minimal catenoids} % stability
\label{sec:proof of BSE theorem}

In this section we will prove Theorem \ref{thm:BSE theorem}.
Let $\Sigma$ be a complete minimal surface immersed in a
complete Riemannian $3$-manifold $M$, and let $\Omega$ be any
subdomain of $\Sigma$.
Recall that a \emph{Jacobi field} on $\Omega \subset \Sigma$
is a $C^\infty$ function $\phi$ such that $\Lcal\phi = 0$
on $\Omega$.

According to Theorem~\ref{thm:FCS80}, in order to show that
a complete minimal surface $\Sigma\subset{}M$ is stable,
we just need to find a positive Jacobi field on $\Sigma$.
On the other hand, if a Jacobi field on $\Sigma$ changes its
sign between interior and exterior of a compact subdomain
$\Omega$ of $\Sigma$ and vanishes on $\partial\Omega$,
we can conclude that $\Omega$ is a maximally weakly stable
minimal surface, which also implies that $\Sigma$ is unstable.

The geometry of the ambient space provides useful Jacobi fields.
More precisely, we have the following classical results.

\begin{theorem}[{\cite[pp. 149--150]{Xin03}}]
\label{thm:Killing Jacobi field}
Let $\Sigma$ be a complete minimal surface immersed in
complete Riemannian $3$-manifold $M$
and let $V$ be a Killing field on $M$. The function
$\zeta = \inner{V}{N}$, given by the inner product in
$M$ of the Killing field $V$ with the unit normal $N$ to
the immersion, is a Jacobi field on $\Sigma$.
\end{theorem}

\begin{theorem}[{\cite[Theorem 2.7]{BdC80b}}]
\label{thm:BdC}
Let $X(a,\cdot):\Sigma\to{}M$ be a $1$-parameter family of
minimal immersions. Then, for each fixed $a_0$, the function
\begin{equation*}
  \xi=\biginner{\ppl{X}{a}(a_0,\cdot)}{N}
\end{equation*}
is a Jacobi field on $X(a_0,\Sigma)$, where
$\inner{\cdot}{\cdot}$ is the inner product in $M$ and
$N$ is the unit normal vector field on the minimal surface
$X(a_0,\Sigma)$.
\end{theorem}

\subsection{Jacobi fields on spherical catenoids}
Next we will follow B{\'e}rard and Sa Earp
\cite{BSE10} to introduce the vertical Jacobi fields and
the variation Jacobi fields on the minimal spherical catenoids
$\{\CC_a\}_{a>0}$ in $\B^3$, which
will be used to prove Theorem~\ref{thm:BSE theorem}.

%%%%%%%%%%%%%%%%%%%%%%%%%%%%%
Recall that the semi disk $\B_{+}^2$ is equipped with the metric
\eqref{eq:warped product metric}, it's easy to get
the arc length of the catenary $\sigma_a$:
\begin{equation*}
  s(a,t)
   =\int_{a}^{t}\frac{\sinh(2\tau)}
   {\sqrt{\cosh^2(2\tau)-\cosh^2(2a)}}\,d\tau
   =\frac{1}{2}\,\cosh^{-1}\left(\frac{\cosh(2t)}{\cosh(2a)}\right)\ ,
   \quad t\geq{}a\ .
\end{equation*}
For any $s\in(-\infty,\infty)$, let
\begin{align}\label{eq:x(a,s)}
  x(a,s)&=\sqrt{2}\,\sinh(2a)
          \int_{0}^{s}\frac{\sqrt{\cosh(2a)\cosh(2t)-1}}
          {\cosh^{2}(2a)\cosh^{2}(2t)-1}\,dt\ ,\\
          \label{eq:y(a,s)}
  y(a,s)&=a+\int_{0}^{s}\frac{\cosh(2a)\sinh(2t)}
           {\sqrt{\cosh^{2}(2a)\cosh^{2}(2t)-1}}\,dt\\
        &=\frac{1}{2}\,\cosh^{-1}(\cosh^{2}(2a)\cosh^{2}(2s))\ .
\end{align}
It's easy to verify that
\begin{equation}
   x(a,s)=\rho(a,y(a,s))
\end{equation}
for $s\geq{}0$ and that the function
\begin{equation}
   s\mapsto(x(a,s),y(a,s))
\end{equation}
is arc-length parametrization of the catenary $\sigma_a$
for $s\in(-\infty,\infty)$, where $\rho(\cdot,\cdot)$ is
given by \eqref{eq:catenary(a,t)}
(see \cite[Proposition 4.2]{BSE10}).

Just as in \eqref{eq:(u,v) in terms of (x,y)}, we define
\begin{equation*}
   u(a,s)=\frac{\sinh{}x\cosh{}y}{1+\cosh{}x\cosh{}y}
         \quad\text{and}\quad
   v(a,s)=\frac{\sinh{}y}{1+\cosh{}x\cosh{}y}\ ,
\end{equation*}
where $x=x(a,s)$ and $y=y(a,s)$.
The parametric equation of the catenoid
$\CC_a$ in $\B^3$ is given by
\begin{equation}\label{eq:position vector of catenary}
  Y(a,s,\theta)=
     \begin{pmatrix}
        u\\
        v\omega_\theta
     \end{pmatrix}\ ,
     \quad
     s\in\R\ \text{and}\
     \theta\in[0,2\pi]\ ,
\end{equation}
where $\omega_\theta=\begin{pmatrix}
        \cos\theta\\
        \sin\theta
     \end{pmatrix}$.
Direct computation shows that the unit normal vector
of the catenoid $\CC_a$ at $Y(a,s,\theta)$ is
\begin{equation}
  N(a,s,\theta)=
     \begin{pmatrix}
        v_{s}\\
        -u_{s}\omega_\theta
     \end{pmatrix}\ ,
     \quad
     s\in\R\ \text{and}\
     \theta\in[0,2\pi]\ ,
\end{equation}
where $u_s$ and $v_s$ are the partial derivatives
of $u$ and $v$ on $s$ respectively.

\begin{definition}Let $V$ be the Killing vector field associated
with the hyperbolic translations along the geodesic
$t\mapsto(\tanh(t/2),0,0)\in\B^3$.
The \emph{vertical Jacobi field} on the
catenoid $\CC_a$ is the function
\begin{equation}\label{eq:vertical Jacobi field}
  \zeta(a,s)=\inner{V(a,s,\theta)}{N(a,s,\theta)}\ ,
\end{equation}
where $V(a,s,\theta)$ is the restriction of the Killing vector field
$V$ to the minimal catenoid $\CC_a$ defined by
\eqref{eq:position vector of catenary}.

The \emph{variation Jacobi field} on the
catenoid $\CC_a$ is
\begin{equation}\label{eq:variation Jacobi field}
  \xi(a,s)=-\inner{Y_a(a,s,\theta)}{N(a,s,\theta)}\ ,
\end{equation}
where $Y_a=\ppl{Y}{a}$.
\end{definition}

In order to find the detail expressions of the vertical
and the variation Jacobi fields on the catenoids, we need
some notations (see \cite[$\S$4.2]{BSE10}). Let
\begin{equation}\label{eq:def of f(a,s)}
  f(a,s)=\frac{\sinh^{2}(2a)\cosh(2s)}{\cosh^{2}(2a)\cosh^{2}(2s)-1}\ ,
\end{equation}
and let
\begin{equation}\label{eq:def of I(a,t)}
  I(a,t)=\frac{n(\cosh(2a),\cosh(2t))}
  {d(\cosh(2a),\cosh(2t))}
\end{equation}
where
\begin{itemize}
  \item $n(A,T)=A(3-A^2)T^2+(A^2-1)T-2A$, and
  \item $d(A,T)=(AT+1)^2(AT-1)^{3/2}$.
\end{itemize}
For the functions $x(a,s)$ and $y(a,s)$ given by \eqref{eq:x(a,s)}
and \eqref{eq:y(a,s)}, the notations
$x_a$, $x_s$, $y_a$ and $y_s$ denote the partial derivatives of
$x(a,s)$ and $y(a,s)$ on $a$ and $s$
respectively.

\begin{proposition}[{\cite[$\S$4.2.1]{BSE10}}]
The vertical Jacobi field $\zeta(a,s)$ is given by
\begin{equation*}
  \zeta(a,s)
     = \sqrt{2}\,\cosh(y(a,s))y_{s}(a,s)
     = \frac{\cosh(2a)\sinh(2s)}{\sqrt{\cosh(2a)\cosh(2s)-1}}\ .
\end{equation*}
The variation Jacobi field $\xi(a,s)$ is given by
\begin{align*}
  \xi(a,s)
     &=-\cosh(y(a,s))(x_a(a,s)y_s(a,s)-x_s(a,s)y_a(a,s)) \\
     &=f(a,s)-\zeta(a,s)\int_{0}^{s}I(a,t)dt\ ,
\end{align*}
where $f(a,s)$ and $I(a,t)$ are given by
\eqref{eq:def of f(a,s)} and \eqref{eq:def of I(a,t)}
respectively.
\end{proposition}

%---------------------------------------------------------

Since $x(a,\infty)$ is well defined for any $a>0$, we may set
\begin{equation}
  E(a)=\ddl{}{a}\,x(a,\infty)
      =\sqrt{2}\int_{0}^{\infty}I(a,t)dt\ .
\end{equation}
Equivalently we have the following identity
(see \cite[p. 3665]{BSE10}):
\begin{equation}\label{eq:E=d0'}
  E(a)=\frac{d_{0}'(a)}{\sqrt{2}}\ ,
\end{equation}
where $d_{0}'(a)$ is derivative of the function $d_{0}(a)$
given by \eqref{eq:Gomes function I}.

For any (connected) interval $\mathbf{I}\subset\R$, we define
\begin{equation}
  \CC(a,\mathbf{I})=\{Y(a,s,\theta)\in\B^3\ |\ s\in{}\mathbf{I}
  \ \text{and}\ \theta\in[0,2\pi]\}\ ,
\end{equation}
where $Y(a,s,\theta)$ is given by \eqref{eq:position vector of catenary}.

\begin{lemma}[{\cite[Lemma 4.5]{BSE10}}]
\label{lem:BSE-stability-half-catenoid}
For any constant $a>0$, the half catenoids $\CC(a,(-\infty,0])$
and $\CC(a,[0,\infty))$ are both stable.

Any Jacobi field $\eta(a, s)$ depending only on the radial
variable s on $\CC_a$ can change its sign at most
once on either $(-\infty,0]$ or $[0,\infty)$.
\end{lemma}

\begin{proof}The first part follows from the fact that $\zeta(a,s)$
doesn't change its sign on either $\CC(a,(-\infty,0])$ or
$\CC(a,[0,\infty))$ and $\Lcal{}\zeta=0$.

Assume that some Jacobi field $\eta(a,s)$ on $\CC_a$ changes its
sign more than once on $[0,\infty)$, then $\eta(a,s)$ has more than
two zeros on $[0,\infty)$, say
$0<{}z_1<z_2<\cdots$. Let $\mathbf{I}=[z_1,z_2]$ and let
$\phi(a,s)$ be the restriction of $\eta(a,s)$ to
$\CC(a,\mathbf{I})$, then we have
$\phi\in{}C_{0}^\infty(\CC(a,\mathbf{I}))$ and $\Lcal{}\phi=0$,
which imply that $\lambda_{1}(\CC(a,\mathbf{I}))\leq{}0$.
This is a contradiction, since $\CC(a,\mathbf{I})$ is a compact
connected subdomain of $\CC(a,[0,\infty))$, which must be stable.
\end{proof}

The following theorem, whose proof can be found
in \cite[p. 3663]{BSE10}, is crucial to the proof of
Theorem~\ref{thm:BSE theorem}. For convenience of the reader, we
rephrase the original proof here. Because of \eqref{eq:E=d0'},
\eqref{eq:derivative of d0} and
Lemma~\ref{lem:derivative of d0}, we always have
$\cosh^{2}(2a)<3$ if $E(a)=0$, hence we also
simplify the proof of \cite[Theorem 4.7 (1)]{BSE10}.

\begin{theorem}[{\cite[Theorem 4.7 (1)]{BSE10}}]
\label{thm:BSE-stability}
Let $\sigma_a$ be the catenary given by
\eqref{eq: catenary equation} and let $\CC_a$ be the
minimal surface of revolution along the $u$-axis
whose generating curve is the catenary $\sigma_a$.
\begin{enumerate}
  \item If $E(a)\leq{}0$, then $\CC_a$ is stable.
  \item If $E(a)>0$, then $\CC_a$ is unstable
        and has index $1$.
\end{enumerate}
\end{theorem}

\begin{proof}(1) As state in Lemma \ref{lem:BSE-stability-half-catenoid},
the function $\xi(a, s)$ can change its sign at most
once on $(0,\infty)$ and $(-\infty,0)$ respectively. Observe
that the function $\xi(a, s)$ is even and that $\xi(a,0)=1$. To determine
whether $\xi$ has a zero, it suffices to look at its behaviour at infinity.

If $E(a)<0$, then $\int_{0}^{\infty}I(a,t)dt<0$, which implies that
$\xi(a,s)\to\infty$ as $s\to\pm\infty$, therefore $\xi(a,s)>0$
for all $s\in(-\infty,\infty)$.

If $E(a)=0$, we have the following equation
\begin{equation*}
  \xi(a,s)=f(a,s)+\zeta(a,s)\int_{s}^{\infty}I(a,t)dt\ .
\end{equation*}
By \eqref{eq:E=d0'}, \eqref{eq:derivative of d0} and
Lemma~\ref{lem:derivative of d0}, we can see that if $E(a)=0$, then
\begin{equation*}
   \cosh^{2}(2a)<\left(\frac{1+\sqrt{5}}{2}\right)^2<3\ ,
\end{equation*}
and so $I(a,t)>0$ if $t$ is sufficient large.
As $s$ is sufficiently large, $\xi(a,s)>0$,
thus $\xi(a,s)>0$ for all $s\in(-\infty,\infty)$. Therefore
$\CC_{a}$ is stable if $E(a)\leq{}0$.

(2) Recall that the variation Jacobi field  $\xi(a,s)$ can
change its sign at most once on either $(0,\infty)$ or $(-\infty,0)$
by Lemma~\ref{lem:BSE-stability-half-catenoid}.
Now suppose that $E(a)>0$, since $\xi(a,0)=1$ and $\xi(a,s)\to-\infty$
as $s\to\pm\infty$, we know that
$\xi(a,s)$ has exactly two symmetric zeros in $(-\infty,\infty)$,
which are denoted by $\pm{}z(a)$.
Let $\CC(z(a))$ be the subdomain of $\CC_a$ defined by
\begin{equation}\label{eq:portion of Pi}
  \CC(z(a))=\CC(a,[-z(a),z(a)])\ .
\end{equation}
Let $\phi(a,s)$ be restriction of $\xi(a,s)$ to $\CC(z(a))$, then
$\phi\in{}C_{0}^\infty(\CC(z(a)))$ and $\Lcal{}\phi=0$. This implies
that $\lambda_{1}(\CC(z(a)))\leq{}0$, which can imply that any compact
connected subdomain of $\CC_a$ containing $\CC(z(a))$ must be
unstable by Lemma~\ref{lem:monotonicity of eigenvalue}. Therefore
$\CC_a$ has index at least one.
By \cite[Theorem 4.3]{Seo11} or \cite[$\S$~3.3]{Tuz93},
$\CC_a$ has index one.
\end{proof}

\subsection{The final step to prove Theorem~\ref{thm:BSE theorem}}
In order to prove Theorem~\ref{thm:BSE theorem},
we still need two lemmas.

\begin{lemma}\label{lem:derivative of d0}
Let $\phi(a,t)=\sqrt{5}\,\cosh(a+t)-\cosh(3a+t)$,
then $\phi(a,t)\leq{}0$ for
$(a,t)\in{}[A_3,\infty)\times[0,\infty)$,
where the constant $A_{3}$ is defined by
\begin{equation}\label{eq:constant Lambda-3}
   A_3=\cosh^{-1}\left(\frac{\sqrt{3+\sqrt{5}}}{2}\right)
            \approx 0.530638\ .
\end{equation}
\end{lemma}

\begin{proof}It's easy to verify that $\phi(a,t)\leq{}0$ is
equivalent to
\begin{equation*}
   \cosh(3a)-\sqrt{5}\cosh{}a+\tanh{}t\cdot
   (\sinh(3a)-\sqrt{5}\sinh{}a)\geq{}0\ .
\end{equation*}
Since $\tanh{}t\geq{}0$ for $t\geq{}0$ and
$\sinh(3a)-\sqrt{5}\sinh{}a\geq{}0$ for $a\geq{}0$,
we need solve the inequality
$\cosh(3a)-\sqrt{5}\cosh{}a\geq{}0$.

Let $A_3$ be the solution of the equation
$0=\sqrt{5}\,\cosh{}a-\cosh(3a)=
(\sqrt{5}-(4\cosh^{2}a-3))\cosh{}a$,
then $\phi(a,t)\leq{}0$
if $a\geq{}A_3$ and $t\geq{}0$.
\end{proof}

\begin{lemma}\label{lem:second derivative of d[lambda]}
Let $\psi(a,t)$ be the function given by
\begin{equation}\label{eq:psi(t,l)}
\begin{aligned}
   \psi(a,t)=
        &\,76\sinh(2a)-22\sinh(2t)+29\sinh(4a+2t)\\
        &\,+\sinh(8a+2t)-26\sinh(6a+4t)-6\sinh(10a+4t)\\
        &\,-25\sinh(8a+6t)+\sinh(12a+6t)\ .
\end{aligned}
\end{equation}
Then $\psi(a,t)<0$ for all $(a,t)\in{}[0,A_4]\times[0,\infty)$,
where the constant
\begin{equation}\label{eq:constant Lambda-4}
   A_4=\frac{1}{4}\cosh^{-1}\left(\frac{35+\sqrt{1241}}{8}\right)
            \approx 0.715548
\end{equation}
is the solution of the equation
$4\cosh^2(4a)-35\cosh(4a)-1=0$.
\end{lemma}

\begin{proof}Expand each term in $\psi(a,t)$ with the form
$\sinh(ma+nt)$ , then we may write
$\psi(a,t)=\psi_1(a,t)+\psi_2(a,t)$, where
\begin{equation*}
\begin{aligned}
   \psi_1(a,t)=
        &\,-22\sinh(2t)+29\sinh(2t)\cosh(4a)+\sinh(2t)\cosh(8a)\\
        &\,-26\sinh(4t)\cosh(6a)-6\sinh(4t)\cosh(10a)\\
        &\,-25\sinh(6t)\cosh(8a)+\sinh(6t)\cosh(12a))\ ,
\end{aligned}
\end{equation*}
and
\begin{equation*}
\begin{aligned}
   \psi_2(a,t)=
        &\,76\sinh(2a)+29\cosh(2t)\sinh(4a)+\cosh(2t)\sinh(8a)\\
        &\,-26\cosh(4t)\sinh(6a)-6\cosh(4t)\sinh(10a)\\
        &\,-25\cosh(6t)\sinh(8a)+\cosh(6t)\sinh(12a)\ .
\end{aligned}
\end{equation*}

\noindent\textbf{Claim}: \emph{$\psi_1(a,t)\leq{}0$ and
$\psi_2(a,t)\leq{}0$ for $(a,t)\in[0,A_4]\times[0,\infty)$}.

\begin{proof}[\bf{Proof of Claim}]
First of all, we will show that $\psi_1(a,\cdot)\leq{}0$
for $a\in[0,A_4]$. Since $\cosh(2t)\geq{}1$ for
any $t\in[0,\infty)$, we have the estimate
\begin{align*}
   \psi_1(a,t)
       =&\,-\sinh(2t)(22-29\cosh(4a)-\cosh(8a)\\
        &\qquad\qquad +52\cosh(2t)\cosh(6a)
                      +12\cosh(2t)\cosh(10a))\\
        &\,-\sinh(6t)(25\cosh(8a)-\cosh(12a))\\
    \leq&\,-\sinh(2t)(22-29\cosh(4a)-\cosh(8a)\\
        &\qquad\qquad +52\cosh(6a)+12\cosh(10a))\\
        &\,-\sinh(6t)(25\cosh(8a)-\cosh(12a))\ .
\end{align*}
Since $52\cosh(6a)-29\cosh(4a)>0$ and
$12\cosh(10a)-\cosh(8a)>0$ for $0\leq{}a<\infty$ and
$25\cosh(8a)-\cosh(12a)>0$ for $0\leq{}a\leq{}A_4$,
we have $\psi_1(a,\cdot)<0$ for $0\leq{}a\leq{}A_4$.

Secondly, for any $a\geq{}0$, we apply the inequality
\begin{equation*}
   \sinh((m+n)a)\geq\sinh(ma)+\sinh(na)\ ,
\end{equation*}
where $m,n$ are positive integers, to get the following
inequalities
\begin{itemize}
   \item $\sinh(6a)\geq\sinh(4a)+\sinh(2a)$,
   \item $\sinh(8a)\geq{}4\sinh(2a)$, and
   \item $\displaystyle\sinh(10a)\geq
          \begin{cases}
             \sinh(8a)+\sinh(2a)\\
             \sinh(4a)+3\sinh(2a)\\
             5\sinh(2a)
          \end{cases}$\ ,
\end{itemize}
which can imply the estimate
\begin{align*}
   \psi_2(a,t)\leq
        &\,-46\sinh(2a)(\cosh(4t)-1)-30\sinh(2a)(\cosh(6t)-1)\\
        &\,-29\sinh(4a)(\cosh(4t)-\cosh(2t))\\
        &\,-\sinh(8a)(\cosh(4t)-\cosh(2t))\\
        &\,-\cosh(6t)\left(\frac{35}{2}
            \sinh(8a)-\sinh(12a)\right)\ .
\end{align*}
Since $\frac{35}{2}\sinh(8a)-\sinh(12a)=
\sinh(4a)(1+35\cosh(4a)-4\cosh^2(4a))\geq{}0$ if
$0\leq{}a\leq{}A_4$ and the fact
$\cosh(6t)\geq\cosh(4t)\geq\cosh(2t)\geq{}1$ for $0\leq{}t<\infty$,
we have $\psi_2(a,t)\leq{}0$ for
$(a,t)\in[0,A_4]\times[0,\infty)$.
\end{proof}

Therefore $\psi(a,t)<0$ for $(a,t)\in[0,A_4]\times[0,\infty)$.
\end{proof}

Now we are able to prove Theorem~\ref{thm:BSE theorem}.

\begin{theorem}[\cite{BSE10}]\label{thm:BSE theorem}
There exists a constant $a_c\approx{}0.49577389$ such that
the following statements are true:
\begin{enumerate}
     \item $\CC_a$ is an unstable minimal surface with index
           one if $0<a<a_c$;
     \item $\CC_a$ is a globally stable minimal surface if
           $a\geq{}a_c$.
  \end{enumerate}
\end{theorem}

\begin{proof}[\bf{Proof of Theorem~\ref{thm:BSE theorem}}]
Recall that we have $E(a)=d_{0}'(a)/\sqrt{2}$ by
\eqref{eq:E=d0'}. We claim that $d_{0}'(a)$
satisfies the following conditions:
\begin{itemize}
   \item $d_{0}'(a)\to\infty$ as $a\to{}0^{+}$ and
         $d_{0}'(a)<0$ on $[A_3,\infty)$, and
   \item $d_{0}'(a)$ is decreasing on $(0,A_4)$,
\end{itemize}
where $A_3<A_4$ are constants defined in
Lemma \ref{lem:derivative of d0} and
Lemma \ref{lem:second derivative of d[lambda]}.
These conditions can imply that $d_{0}'$ has a unique zero
$a_c\in(0,\infty)$ such that $d_{0}'(a)>0$ if $0<a<a_c$ and
$d_{0}'(a)<0$ if $a>a_c$, hence together with
Theorem~\ref{thm:BSE-stability}, the theorem follows
(see Fig.~\ref{fig:derivative of Gomes function}).

%===================================================================
\begin{figure}[htbp]
  \begin{center}
     \includegraphics[scale=0.85]{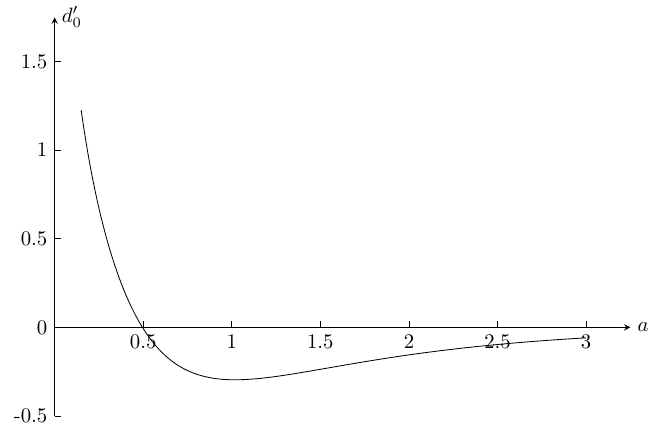}
  \end{center}
  \caption{The derivative of the function $d_{0}(a)$ for
  $a\in(0,3]$.}
  \label{fig:derivative of Gomes function}
\end{figure}
%===================================================================

Next let's prove the above Claim. It's easy to verify that
\begin{equation}\label{eq:derivative of d0}
   d_{0}'(a)=\int_{0}^{\infty}
      \frac{\sinh(a+t)
      (5\cosh^2(a+t)-\cosh^{2}(3a+t))}
      {\cosh^2(a+t)\sqrt{\sinh(2t)\sinh^{3}(4a+2t)}}\,dt\ .
\end{equation}
By Lemma \ref{lem:derivative of d0}, $d_0'(a)<0$ on
$(A_3,\infty)$. Now let
\begin{equation}
   h(a,t)=\frac{\sinh(a+t)
      (5\cosh^2(a+t)-\cosh^{2}(3a+t))}
      {\cosh^2(a+t)\sqrt{\sinh(2t)\sinh^{3}(4a+2t)}}\ .
\end{equation}
Then for any fixed constant $a>0$, we have the estimates
\begin{equation}
   h(a,t)\sim{}C_{1}(a)\left(
   \sqrt{\sinh{}t}+\frac{\cosh{}t}{\sqrt{\sinh{}t}}\right)\ ,
   \quad\text{as}\ t\to{}0\ ,
\end{equation}
and
\begin{equation}
   h(a,t)\sim
   \frac{C_{2}(a)}{\cosh{}t\cosh(2a+t)\sinh(4a+2t)}\ ,
   \quad\text{as}\ t\to\infty\ .
\end{equation}
Hence $d_{0}'(a)$ is well defined for $a>0$, and then
\begin{align*}
   \lim_{a\to{}0^{+}}d_{0}'(a)
    &=\int_{0}^{\infty}\frac{1}{\sinh{}t\cosh^{2}t}\,dt\\
    &=\left[\log\left(\frac{\cosh{}t-1}{\cosh{}t+1}\right)+
      \frac{1}{\cosh{}t}\right]_{t=0}^{t=\infty}=\infty\ .
\end{align*}
Next, we have
\begin{equation}\label{eq:2nd derivative of d0}
   d_{0}''(a)=\int_{0}^{\infty}\frac{\psi(a,t)}
   {16\cosh^3(a+t)\sqrt{\sinh(2t)\sinh^{5}(4a+2t)}}\,dt\ ,
\end{equation}
here $\psi(a,t)$ is the function defined  by \eqref{eq:psi(t,l)}.
By the result in Lemma \ref{lem:second derivative of d[lambda]},
$d_{0}''(a)<0$ for $a\in(0,A_4)$, thus
$d_{0}'(a)$ is decreasing on $(0,A_4)$.
\end{proof}

%===================================================================
\section{Least area minimal catenoids}
\label{sec:proof of main theorem}

In this section, we will prove Theorem \ref{thm:main theorem}.
First of all we need some results for proving
Theorem~\ref{thm:main theorem}.

\begin{proposition}[{\cite[Proposition 4.8 and Lemma 4.9]{BSE10}}]
\label{thm:Berard-Sa Earp}
Let $\sigma_a\subset\B_{+}^2$ be the catenary given by
\eqref{eq: catenary equation}. For $0<a_1<a_2$,
the catenaries $\sigma_{a_1}$ and $\sigma_{a_2}$
intersect at most at two symmetric points and they do so if and
only if $d_{0}(a_1)<d_{0}(a_2)$. Furthermore we have the following
results:
\begin{enumerate}
   \item For $a_1,a_2\in(0,a_c)$, the catenaries
         $\sigma_{a_1}$ and $\sigma_{a_2}$ intersect exactly
         at two symmetric points {\rm(}see
         Fig.~\ref{fig:unstable catenaries}{\rm)}.
   \item All catenaries $\{\sigma_a\}_{a\geq{}a_c}$ foliate
         the subdomain of $\B^2_{+}$ which is bounded by the
         catenary $\sigma_{a_c}$ and the arc of
         $\partial_{\infty}\B^2_{+}$ between
         the asymptotic boundary points of $\sigma_{a_c}$
         {\rm(}see Fig.~\ref{fig:stable catenaries}{\rm)}.
\end{enumerate}
\end{proposition}

\begin{figure}[htbp]
  \begin{center}
     \includegraphics[scale=0.9]{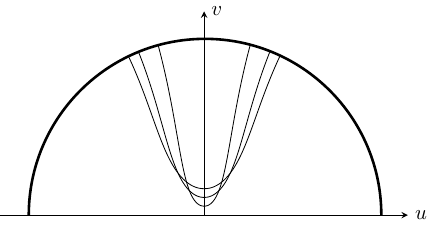}
  \end{center}
  \caption{Catenaries $\sigma_a$ for $a=0.1$, $0.2$ and $0.3$.
  The catenaries in the family $\{\CC_a\}_{0<a<a_c}$ intersect with
  each other at exactly two symmetric points.}
  \label{fig:unstable catenaries}
\end{figure}

\begin{figure}[htbp]
  \begin{center}
     \includegraphics[scale=0.9]{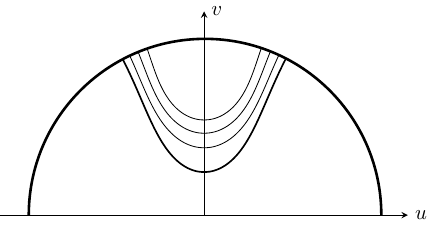}
  \end{center}
  \caption{Catenaries $\sigma_a$ for $a=a_c$, $0.8$,
  $1.0$ and $1.2$. All catenaries in $\{\sigma_a\}_{a\geq{}a_c}$
  foliate the subdomain of the semi-disk $\B^2_{+}$ above the
  catenary $\sigma_{a_c}$.}
  \label{fig:stable catenaries}
\end{figure}

According Proposition~\ref{thm:Berard-Sa Earp}, the catenaries
$\sigma_{a_1}$ and $\sigma_{a_2}$ intersect exactly at two points
if $0<a_1<a_2<a_c$. In order to prove
Theorem~\ref{thm:main theorem}, we require that the intersections
of $\sigma_{a_1}$ and $\sigma_{a_2}$ should not be contained in
the region of $\B^{2}_{+}$ foliated by $\{\sigma_a\}_{a\geq{}a_c}$.
More precisely, we have the following result.

\begin{proposition}
\label{thm:small maximally weakly stable domain}
For any constant $a\in(0,a_c)$,
\begin{equation*}
  \CC(z(a))\cap\Bigg(\bigcup_{\alpha\geq{}a_c}
  \CC_{\alpha}\Bigg)=\emptyset\ ,
\end{equation*}
where $\CC(z(a))$ is the maximally weakly stable subdomian
of $\CC_a$ which is defined by \eqref{eq:portion of Pi}.
\end{proposition}

\begin{lemma}\label{lem:negative second order derivative of d(a,t)}
Let $R_3$ and $R_4$ be the regions defined by
\begin{align*}
  R_3&=\{(a,t)\in\R^2\ |\ t\geq{}a\geq{}A_3\}\ ,\\
  R_4&=\{(a,t)\in\R^2\ |\ 0<a\leq{}A_4\ \text{and}\ t\geq{}a\}\ ,
\end{align*}
where $A_3$ and $A_4$ are the constants defined in
\eqref{eq:constant Lambda-3} and \eqref{eq:constant Lambda-4}.
Then we have
\begin{equation*}
  \ppl{}{a}\,\rho(a,t)<0
  \quad\text{for}\ (a,t)\in{}R_3\ ,
\end{equation*}
and
\begin{equation*}
  \ppz{}{a}\,\rho(a,t)<0
  \quad\text{for}\ (a,t)\in{}R_4\ ,
\end{equation*}
where $\rho(a,t)$ is defined in \eqref{eq:catenary(a,t)}.
\end{lemma}

\begin{proof}Using the substitution $\tau\mapsto\tau+a$, we have
\begin{equation*}
  \rho(a,t)=\int_{0}^{t-a}\frac{\sinh(2a)}{\cosh(a+\tau)}
  \frac{d\tau}{\sqrt{\sinh^2(2a+2\tau)-\sinh^2(2a)}}\ ,
  \quad{}t\geq{}a\ .
\end{equation*}
Direct computation shows
\begin{align*}
  \ppl{}{a}\,\rho(a,t)=
     &\,\int_{0}^{t-a}\frac{\sinh(a+\tau)
      (5\cosh^2(a+\tau)-\cosh^{2}(3a+\tau))}
      {\cosh^2(a+\tau)\sqrt{\sinh(2\tau)\sinh^{3}(4a+2\tau)}}\,
      d\tau\\
     &-\frac{\sinh(2a)}
      {\cosh{}t\cdot\sqrt{\sinh^{2}(2t)-\sinh^{2}(2a)}}
\end{align*}
and
\begin{align*}
   \ppz{}{a}\,\rho(a,t)=
       &\,\int_{0}^{t-a}\frac{\psi(a,\tau)}
          {16\cosh^3(a+\tau)\sqrt{\sinh(2\tau)\sinh^{5}(4a+2\tau)}}
          \,d\tau\\
       &\,-\frac{\sinh{}t\cdot{}w(a,t)}
          {\sqrt{2}\,\cosh^{2}t\cdot(\cosh(4t)-\cosh(4a))^{3/2}}\ ,
\end{align*}
where $\psi(a,\tau)$ is given by \eqref{eq:psi(t,l)} and $w(a,t)$ is defined by
\begin{align*}
  w(a,t)=&\,-5\sinh(2a)+\sinh(6a)-7\sinh(2a-4t)\\
         &\,-12\sinh(2a-2t)+4\sinh(2a+2t)+\sinh(2a+4t)\ .
\end{align*}
Recall that $t\geq{}a>0$, we have $w(a,t)\geq\sinh(6a)>0$, together with
the arguments in the proofs of Lemma \ref{lem:derivative of d0}
and Lemma \ref{lem:second derivative of d[lambda]}, the proof
of the lemma is complete.
\end{proof}

\begin{proof}[\bf{Proof of Proposition~\ref{thm:small maximally weakly stable domain}}]
Let $\mathcal{D}=\cup_{\alpha\geq{}a_c}\sigma_\alpha$ be
the subregion of $\B_{+}^2$.
We claim that $\sigma_{a_1}\cap\sigma_{a_2}$ is disjoint from
$\mathcal{D}$ for $0<a_1<a_2<a_c$
(see Fig.~\ref{fig:catenary_lambda_shadow}).

%==================================================================
\begin{figure}[htbp]
  \begin{center}
     \includegraphics[scale=0.9]{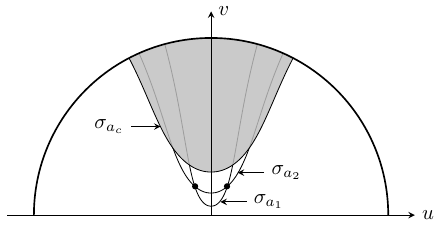}
  \end{center}
  \caption{$\mathcal{D}$ is the shadow region. In this figure,
  $a_1=0.1$ and $a_2=0.25$. We can see that
  $\sigma_{a_1}\cap\sigma_{a_2}$ is disjoint from
  $\mathcal{D}$.}
  \label{fig:catenary_lambda_shadow}
\end{figure}
%===================================================================

Actually if $\sigma_{a_1}\cap\sigma_{a_2}\subset{}\mathcal{D}$,
since $\mathcal{D}$ is foliated
by the catenaries $\{\sigma_\alpha\}_{\alpha\geq{}a_c}$, there
exists $a_3>a_c$ such that $\sigma_{a_1}$, $\sigma_{a_2}$
and $\sigma_{a_3}$ intersect at the same points. Let $t_0$ be the
$y$-coordinate of the intersection points (recall that we equip
$\B^{2}_{+}$ with the warped product metric
\eqref{eq:warped product metric}), then $t_0>a_3$ (see the proof
of Lemma 4.9 in \cite{BSE10}). Consider the function
\begin{equation*}
  \varphi(\alpha)=\rho(\alpha,t_0)\ ,
  \quad{}\alpha\in[a_1,a_3]\ .
\end{equation*}
By our assumption, $\varphi(a_1)=\varphi(a_2)=\varphi(a_3)$. By
L'H{\^o}pital's rule,
there exist $a_4\in(a_1,a_2)$ and $a_5\in(a_2,a_3)$ such that
$\varphi'(a_4)=\varphi'(a_5)=0$. By
Lemma~\ref{lem:negative second order derivative of d(a,t)},
$a_5<A_3$ and then $a_5<A_4$. Applying L'H{\^o}pital's rule again,
there exists $a_6\in(a_4,a_5)\subset(a_1,a_3)\cap(0,A_4)$ such that
$\varphi''(a_6)=0$. This is impossible according to
Lemma~\ref{lem:negative second order derivative of d(a,t)}.
Therefore $\sigma_{a_1}\cap\sigma_{a_2}$ is disjoint from
$\mathcal{D}$ for $0<a_1<a_2<a_c$.

For $0<a_1<a_2<a_3<a_c$, let $y_{ij}$ denote the $y$-coordinate of
the intersection of $\sigma_{a_i}$ and $\sigma_{a_j}$
($1\leq{}i<j\leq{}3$), then we claim that $y_{12}<y_{13}<y_{23}$
(see Fig.~\ref{fig:catenary_intersection}). Otherwise, we must have
$y_{23}<y_{13}<y_{12}$, this may imply there exists $a_4\in(a_3,a_c)$
such that $\sigma_{a_1}$, $\sigma_{a_2}$ and $\sigma_{a_4}$ intersect
at the same points. But this is impossible according to the
similar argument as above. Hence
$\lim\limits_{\alpha\to{}a}(\sigma_a\cap\sigma_\alpha)$
exists and is disjoint from $\mathcal{D}$.

By the definition of the variation Jacobi field $\xi$ in
\eqref{eq:variation Jacobi field}, $\xi(a,\cdot)$ changes its sign at the
different sides of $\lim\limits_{\alpha\to{}a}(\CC_a\cap\CC_\alpha)$, hence
\begin{equation*}
   \lim_{\alpha\to{}a}(\CC_a\cap\CC_\alpha)=
   \partial\CC(z(a))\ ,
\end{equation*}
and therefore
$\CC(z(a))\cap(\cup_{\alpha\geq{}a_c}\CC_{\alpha})=\emptyset$
if $a<a_c$.
\end{proof}

%==================================================================
\begin{figure}[htbp]
  \begin{center}
     \includegraphics[scale=0.9]{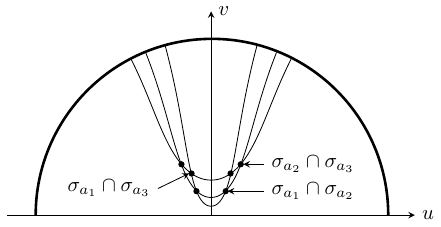}
  \end{center}
  \caption{Recall that each catenary $\sigma_a$ is symmetric about
  the $y$-axis. For $0<a_i<a_j<a_c$, the catenaries $\sigma_{a_i}$
  and $\sigma_{a_j}$ intersect at two points which are also symmetric
  about the $y$-axis. In this figure, $a_1=0.1$, $a_2=0.2$ and
  $a_3=0.4$, and one can see that $y_{12}<y_{13}<y_{23}$.}
  \label{fig:catenary_intersection}
\end{figure}
%===================================================================

The following estimate is crucial to the proof of
Theorem \ref{thm:main theorem}.

\begin{lemma}\label{lem:area compare}
For all real numbers $a>0$, consider the functions
\begin{equation}\label{eq: area difference function}
   f(a)=\int_{a}^{\infty}\sinh{}t\cdot
   \left(\frac{\sinh(2t)}{\sqrt{\sinh^2(2t)-\sinh^2(2a)}}
   -1\right)dt\ ,
\end{equation}
and $g(a)=\cosh{}a-1$, then we have the following results:
\begin{enumerate}
   \item $f(a)$ is well defined for each fixed
         $a\in(0,\infty)$.
   \item $f(a)<g(a)$ for sufficiently large $a$.
\end{enumerate}
\end{lemma}

\begin{proof}(1) Using the substitution $t\to{}t+a$, we have
\begin{equation*}
   f(a)=\int_{0}^{\infty}\sinh(a+t)
   \left(\frac{\sinh(2a+2t)}{\sqrt{\sinh^2(2a+2t)-
   \sinh^2(2a)}}-1\right)dt\ .
\end{equation*}
We will prove that $f(a)<K\cosh{}a$, where
\begin{equation}\label{eq:limit of the ratio}
   K=\int_{0}^{1}\frac{1}{x^2}
          \left(\frac{1}{\sqrt{1-x^4}}-1\right)dx
\end{equation}
is a constant between $0$ and $1$.

Let $\displaystyle\Phi(a,t)=\frac{\sinh(2a+2t)}
{\sqrt{\sinh^2(2a+2t)-\sinh^2(2a)}}$, then
for any fixed $t\in[0,\infty)$, it's easy to
verify that $\Phi(a,t)$ is increasing on $[0,\infty)$
with respect to $a$. So
we have the estimate
\begin{align*}
   \Phi(a,t)   &\leq\lim_{a\to\infty}
                    \frac{\sinh(2a+2t)}
                    {\sqrt{\sinh^2(2a+2t)-\sinh^2(2a)}}\\
               &=\frac{\sinh(2t)+\cosh(2t)}{\sqrt{(\sinh(2t)+
                 \cosh(2t))^2-1}}\\
               &=\frac{e^{2t}}{\sqrt{e^{4t}-1}}
                =\frac{1}{\sqrt{1-e^{-4t}}}\ .
\end{align*}
Besides,
$\sinh(a+t)<(\sinh{}t+\cosh{}t)\cosh{}a=e^{t}\cosh{}a$,
therefore we have the following estimate
\begin{align*}
    f(a)
     &<\cosh{}a\int_{0}^{\infty}e^t\left(\frac{e^{2t}}
       {\sqrt{e^{4t}-1}}-1\right)dt\\
     &=\cosh{}a\int_{0}^{\infty}e^t\left(\frac{1}
       {\sqrt{1-e^{-4t}}}-1\right)dt\\
     &=\cosh{}a\int_{0}^{1}\frac{1}{x^2}
       \left(\frac{1}{\sqrt{1-x^4}}-1\right)dx
       \quad (t\mapsto{}x=e^{-t})
\end{align*}
Since $x^2+1\geq{}1$, we have
\begin{equation*}
\begin{aligned}
   K&=\int_{0}^{1}\frac{1}{x^2}\left(\frac{1}{\sqrt{1-x^4}}-1\right)dx\\
    &<\int_{0}^{1}\frac{1}{x^2}\left(\frac{1}{\sqrt{1-x^2}}-1\right)dx
     =1\ ,
    \end{aligned}
\end{equation*}
where we use the substitution $x\to{}\sin{}x$ to evaluate the second
integral in \eqref{eq:limit of the ratio}.

(2) We have proved that
$f(a)<K\cosh{}a$ for any $a\in[0,\infty)$. Let
\begin{equation}\label{eq:lambda_c}
   a_{l}=\cosh^{-1}\left(\frac{1}{1-K}\right)\ ,
\end{equation}
then $f(a)<g(a)$ if $a\geq{}a_l$.
\end{proof}

\begin{remark}The function $f(a)$ in
\eqref{eq: area difference function} has its geometric meaning:
$2\pi{}f(a)$ is the difference of the \emph{infinite area}
of one half of the catenoid $\CC_a$ and that of the annulus
\begin{equation*}
   \mathcal{A}=\{(0,v,w)\in\B^3\ |\ \tanh(a/2)
   \leq\sqrt{v^2+w^2}<1\}\ .
\end{equation*}
\end{remark}

\begin{remark}
The first definite integral in \eqref{eq:limit of the ratio} is an elliptic
integral. By the numerical computation, $K\approx{}0.40093$, and hence
$a_l\approx{}1.10055$.
\end{remark}

We need the coarea formula that will be used in the proof of
Theorem \ref{thm:main theorem}. The proof of
\eqref{eq:coarea formula I} in Lemma \ref{lem:coarea formula}
can be found in \cite{Wang10}.

\begin{lemma}[Calegari and Gabai\ {\cite[$\S$1]{CG06}}]
\label{lem:coarea formula}
Suppose $\Sigma$ is a surface in the hyperbolic $3$-space
$\B^3$. Let $\gamma\subset\B^3$ be a geodesic,
for any point $q\in\Sigma$, define $\theta(q)$ to be
the angle between the tangent space to $\Sigma$ at $q$, and
the radial geodesic that is through $q$
{\rm(}emanating from $\gamma${\rm)}
and is perpendicular to $\gamma$. Then
\begin{equation}\label{eq:coarea formula I}
   \Area(\Sigma\cap\Nscr_{s}(\gamma))=
   \int_{0}^{s}\int_{\Sigma\cap\partial\Nscr_{t}(\gamma)}
   \frac{1}{\cos\theta}\,dldt\ ,
\end{equation}
where $\Nscr_{s}(\gamma)$ is the hyperbolic $s$-neighborhood
of the geodesic $\gamma$.
\end{lemma}

Now we are able to prove Theorem \ref{thm:main theorem}.

\begin{theorem}\label{thm:main theorem}
There exists a constant $a_l\approx{}1.10055$ defined by \eqref{eq:lambda_c}
such that for any $a\geq{}a_l$ the catenoid $\CC_a$ is a least area
minimal surface in the sense of Meeks-Yau.
\end{theorem}

\begin{proof}[\bf{Proof of Theorem \ref{thm:main theorem}}]
First of all, suppose that $a\geq{}a_{l}$ is an
arbitrary constant.
Suppose that $\partial_{\infty}\CC_{a}=C_1\cup{}C_2$, and
let $P_{i}$ be the geodesic plane asymptotic to $C_{i}$
($i=1,2$). Let $\sigma_{a}=\CC_a\cap\B_{+}^2$
be the generating curve of the
catenoid $\CC_a$.

For $x\in(-d_{0}(a),d_{0}(a))$, let $P(x)$ be the
geodesic plane perpendicular to the $u$-axis such
that $\dist(O,P(x))=|x|$. Now let
\begin{equation}\label{eq:symmetric compact annulus}
   \Sigma=\bigcup_{|x|\leq{}x_{1}}
   \left(\CC_{a}\cap{}P(x)\right)\ ,
\end{equation}
for some $0<x_{1}<d_{0}(a)$.
Let $\partial\Sigma=C_{+}\cup{}C_{-}$. Note that $C_+$ and $C_-$ are
coaxial with respect to the $u$-axis. % or $\gamma_{0}$.

\begin{claim}\label{claim:one}
$\Area(\Sigma)<\Area(P_{+})+\Area(P_{-})$, where
$P_{\pm}$ are the compact subdomains of $P(\pm{}x_{1})$ that are
bounded by $C_{\pm}$ respectively
(see Fig.~\ref{fig:least area catenoid--lambda=1.2}).
\end{claim}

%====================================================================
\begin{figure}[htbp]
  \begin{center}
     \includegraphics[scale=0.9]{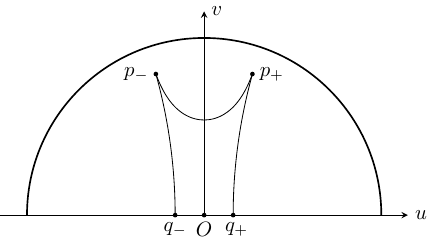}
  \end{center}
  \caption{The curve $\wideparen{p_{-}p_{+}}=\Sigma\cap\B_{+}^2$ is
  the portion of the catenary $\sigma_a$ with $a=1.2$.
  The curves $\wideparen{p_{+}q_{+}}=P_{+}\cap\B_{+}^2$ and
  $\wideparen{p_{-}q_{-}}=P_{-}\cap\B_{+}^2$.
  In this figure, $y_1=2.4$ and $x_1\approx{}0.330439$. By numerical
  computation: $\Area(\Sigma)=54.6636$ and
  $\Area(P_{+}\cup{}P_{-})=57.2643$.}
  \label{fig:least area catenoid--lambda=1.2}
\end{figure}
%====================================================================

\begin{proof}[\bf{Proof of Claim \ref{claim:one}}]
Recall that $P_{\pm}$ are two (totally) geodesic disks with
hyperbolic radius $y_{1}$, so the area of $P_{\pm}$ is given by
\begin{equation}
   \Area(P_{+})=\Area(P_{-})=4\pi\sinh^{2}\left(
   \frac{y_{1}}{2}\right)=2\pi(\cosh{}y_{1}-1)\ ,
\end{equation}
here $(x_1,y_1)\in\sigma_a$ satisfies the equation
\eqref{eq: catenary equation}.

Recall that
$\Area(\Sigma)=\Area(\Sigma\cap\Nscr_{y_{1}}(\gamma_{0}))$,
by the co-area formula we have
\begin{equation}
  \Area(\Sigma)=\int_{a}^{y_{1}}\left(
  \Length(\Sigma\cap\partial\Nscr_{t}(\gamma_{0}))\cdot
  \frac{1}{\cos\theta}\right)dt\ ,
\end{equation}
where the angle $\theta$ is given by \eqref{eq:angle-alpha}, hence
\begin{equation}
  \Area(\Sigma)=\int_{a}^{y_{1}}\left(4\pi\sinh{}t\cdot
  \frac{\sinh(2t)}{\sqrt{\sinh^2(2t)-\sinh^2(2a)}}\right)dt\ .
\end{equation}
By Lemma \ref{lem:area compare}, for any
$a\geq{}a_{l}$ we have
\begin{equation*}
   \int_{a}^{\infty}\sinh{}t\cdot
   \left(\frac{\sinh(2t)}{\sqrt{\sinh^2(2t)-\sinh^2(2a)}}-1\right)dt
   <\cosh{}a-1\ ,
\end{equation*}
therefore for any $y_1\in(a,\infty)$ we have
\begin{equation*}
   4\pi\int_{a}^{y_1}\sinh{}t\cdot\left(\frac{\sinh(2t)}
   {\sqrt{\sinh^2(2t)-\sinh^2(2a)}}-1\right)dt
   <4\pi(\cosh{}a-1)\ ,
\end{equation*}
and then $\Area(\Sigma)<\Area(P_{+})+\Area(P_{-})$.
\end{proof}

\begin{claim}\label{claim:two}
There is no minimal annulus with the same boundary
as that of $\Sigma$ which has smaller area than that of $\Sigma$.
\end{claim}

\begin{proof}[\bf{Proof of Claim \ref{claim:two}}]
Recall that $a$ is an
arbitrary constant chosen to be $\geq{}a_{l}$.
Let $\Omega$ be the subregion of $\B^3$ bounded by
$P(-x_{1})$ and $P(x_{1})$, and let $\T_a$ be the simply connected
subregion of $\B^3$ bounded  by $\CC_a$.

Assume that $\Sigma'$ is a least area annulus with the same boundary
as that of $\Sigma$, and $\Area(\Sigma')<\Area(\Sigma)$. Since
$\Sigma'$ is a least area annulus, it must be a minimal surface.
By \cite[Theorem 5]{MY1982(t)} and
\cite[Theorem 1]{MY1982(mz)}, $\Sigma'$ must be contained in
$\Omega$, otherwise we can use cutting and pasting technique
to get a minimal surface contained in $\Omega$ that has
smaller area. Furthermore, recall that
$\{\CC_\alpha\}_{\alpha\geq{}a_c}$ locally foliates
$\Omega\subset\B^3$, therefore $\Sigma'$ must be contained
in $\T_a\cap\Omega$ by the Maximum Principle. It's easy to
verify that the boundary of $\T_a\cap\Omega$ is given by
$\partial(\T_a\cap\Omega)=\Sigma\cup{}P_{+}\cup{}P_{-}$.

Now we claim that $\Sigma'$ is symmetric about any geodesic
plane that passes through the $u$-axis, i.e., $\Sigma'$ is
a surface of revolution. Otherwise, using the reflection
along the geodesic planes that pass through the $u$-axis,
we can find another annulus $\Sigma''$ with
$\partial\Sigma''=\partial\Sigma'$ such that either
$\Area(\Sigma'')<\Area(\Sigma')$ or $\Sigma''$ contains
\emph{folding curves} %(see \cite[pp. 418--419]{MY1982(t)})
so that we can find smaller area annulus by
the argument in \cite[pp. 418--419]{MY1982(t)}.
Similarly, $\Sigma'$ is symmetric about the $vw$-plane.

Now let $\sigma'=\Sigma'\cap\B_{+}^{2}$, then $\sigma'$
must satisfy the
equations \eqref{eq:differential equations of Pi} for some constant
$a'>0$, which may imply that $\Sigma'$ is a compact subdomain
of some catenoid $\CC_{a'}$ (see
Fig.~\ref{fig:least area catenoid--lambda=1.2 and 0.14341}).
Obviously $\CC_{a'}\cap\CC_{a}=C_{+}\cup{}C_{-}$.

%====================================================================
\begin{figure}[htbp]
  \begin{center}
     \includegraphics[scale=0.9]{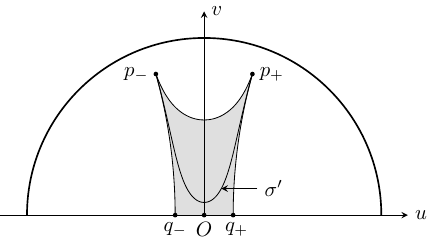}
  \end{center}
  \caption{The shaded region is equal to
  $(\T_a\cap\Omega)\cap\B^{2}_{+}$.
  If $\Sigma'\ne\Sigma$, then $\sigma'=\Sigma'\cap\B_{+}^{2}$ is the
  portion of some catenary $\sigma_{a'}$ with
  $a'<a_c$ and then $\Sigma'$ is unstable.}
  \label{fig:least area catenoid--lambda=1.2 and 0.14341}
\end{figure}
%====================================================================

Since $\Sigma'\subset\T_a\cap\Omega$, we have $a'\leq{}a$.
We claim that if $a'<a$, that is $\Sigma'$ is not the same
as $\Sigma$, then it must be unstable, which implies a contradiction.
In fact, if $a'<a$, since $d_0(a')<d_0(a)<d_{0}(a_c)$,
we have $a'<a_c$ by Proposition~\ref{thm:Berard-Sa Earp}, which
implies that $\CC_{a'}$ is unstable. Besides, according to
Proposition \ref{thm:small maximally weakly stable domain},
the subdomain $\Sigma'$ of $\CC_{a'}$ is also unstable,
so it couldn't be a least area
minimal surface unless $\Sigma'\equiv\Sigma$.
Therefore any compact annulus of the form
\eqref{eq:symmetric compact annulus}
is a least area minimal surface.
\end{proof}

Now let $S$ be any compact domain of $\CC_a$, then we
always can find a compact annulus $\Sigma$ of the form
\eqref{eq:symmetric compact annulus} such that $S\subset\Sigma$.
If $S$ is not a least area minimal surface, then we can
use the cutting and
pasting technique to show that $\Sigma$ is not a least
area minimal surface.
This is contradicted to the above argument.

Therefore if $a\geq{}a_l$, then $\CC_a$ is a least area minimal
surface in the sense of Meeks-Yau.
\end{proof}

\begin{remark}In the proof of \emph{Claim 2} in
Theorem \ref{thm:main theorem}, if $\Sigma'$ is an annulus type
minimal surface but it is not a least area
minimal surface, then it might not be a surface of revolution
(see \cite[p. 234]{Lop00}).
\end{remark}

\begin{corollary}
There exists a finite constant
$D_{l}=2d_0(a_l)\approx{}0.729183$ such that for two disjoint
circles $C_{1},C_{2}\subset{}S_{\infty}^{2}$, if
$d_{L}(C_{1},C_{2})\leq{}D_{l}$, then there exist a least area
spherical minimal catenoid $\CC$ which is asymptotic to
$C_{1}\cup{}C_{2}$.
\end{corollary}

%--------------------------------------------------------------------
%\bibliographystyle{plain}
%\bibliographystyle{amsalpha}
\bibliographystyle{amsplain}
\bibliography{ref_catenoid}

\providecommand{\bysame}{\leavevmode\hbox to3em{\hrulefill}\thinspace}
\providecommand{\MR}{\relax\ifhmode\unskip\space\fi MR }
\providecommand{\MRhref}[2]{%
  \href{http://www.ams.org/mathscinet-getitem?mr=#1}{#2}
}
\providecommand{\href}[2]{#2}
\begin{thebibliography}{10}

\bibitem{BCH09}
Allen Back, Manfredo do~Carmo, and Wu-Yi Hsiang, \emph{On some fundamental
  equations of equivariant {R}iemannian geometry}, Tamkang J. Math. \textbf{40}
  (2009), no.~4, 343--376.

\bibitem{BdC80b}
Jo{\~a}o~Lucas Barbosa and Manfredo do~Carmo, \emph{Stability of minimal
  surfaces in spaces of constant curvature}, Bol. Soc. Brasil. Mat. \textbf{11}
  (1980), no.~1, 1--10.

\bibitem{Bea95}
Alan~F. Beardon, \emph{The geometry of discrete groups}, Graduate Texts in
  Mathematics, vol.~91, Springer-Verlag, New York, 1995, Corrected reprint of
  the 1983 original.

\bibitem{BSE09}
Pierre B{\'e}rard and Ricardo Sa~Earp, \emph{Lindel{\"o}f's theorem for
  catenoids revisited}, arXiv:0907.4294v1 (2009).

\bibitem{BSE10}
\bysame, \emph{Lindel\"of's theorem for hyperbolic catenoids}, Proc. Amer.
  Math. Soc. \textbf{138} (2010), no.~10, 3657--3669.

\bibitem{CG06}
Danny Calegari and David Gabai, \emph{Shrinkwrapping and the taming of
  hyperbolic 3-manifolds}, J. Amer. Math. Soc. \textbf{19} (2006), no.~2,
  385--446.

\bibitem{Can07}
Alberto Candel, \emph{Eigenvalue estimates for minimal surfaces in hyperbolic
  space}, Trans. Amer. Math. Soc. \textbf{359} (2007), no.~8, 3567--3575
  (electronic).

\bibitem{CM11}
Tobias~Holck Colding and William~P. Minicozzi, II, \emph{A course in minimal
  surfaces}, Graduate Studies in Mathematics, vol. 121, American Mathematical
  Society, Providence, RI, 2011.

\bibitem{dOS98}
Geraldo de~Oliveira and Marc Soret, \emph{Complete minimal surfaces in
  hyperbolic space}, Math. Ann. \textbf{311} (1998), no.~3, 397--419.

\bibitem{dCD83}
Manfredo do~Carmo and Marcos Dajczer, \emph{Rotation hypersurfaces in spaces of
  constant curvature}, Trans. Amer. Math. Soc. \textbf{277} (1983), no.~2,
  685--709.

\bibitem{dCGT86}
Manfredo do~Carmo, Jonas de~Miranda Gomes, and Gudlaugur Thorbergsson,
  \emph{The influence of the boundary behaviour on hypersurfaces with constant
  mean curvature in {$H^{n+1}$}}, Comment. Math. Helv. \textbf{61} (1986),
  no.~3, 429--441.

\bibitem{FCS80}
Doris Fischer-Colbrie and Richard Schoen, \emph{The structure of complete
  stable minimal surfaces in {$3$}-manifolds of nonnegative scalar curvature},
  Comm. Pure Appl. Math. \textbf{33} (1980), no.~2, 199--211.

\bibitem{Gom87}
Jonas de~Miranda Gomes, \emph{Spherical surfaces with constant mean curvature
  in hyperbolic space}, Bol. Soc. Brasil. Mat. \textbf{18} (1987), no.~2,
  49--73.

\bibitem{Hsi82}
Wu-yi Hsiang, \emph{On generalization of theorems of {A}. {D}. {A}lexandrov and
  {C}. {D}elaunay on hypersurfaces of constant mean curvature}, Duke Math. J.
  \textbf{49} (1982), no.~3, 485--496.

\bibitem{Wang11a}
Zheng Huang and Biao Wang, \emph{Counting minimal surfaces in quasi-fuchsian
  three-manifolds}, Preprint (2012), to appear in Trans of AMS.

\bibitem{LR85}
Gilbert Levitt and Harold Rosenberg, \emph{Symmetry of constant mean curvature
  hypersurfaces in hyperbolic space}, Duke Math. J. \textbf{52} (1985), no.~1,
  53--59.

\bibitem{Lop00}
Rafael L{\'o}pez, \emph{Hypersurfaces with constant mean curvature in
  hyperbolic space}, Hokkaido Math. J. \textbf{29} (2000), no.~2, 229--245.

\bibitem{MT98}
Katsuhiko Matsuzaki and Masahiko Taniguchi, \emph{Hyperbolic manifolds and
  {K}leinian groups}, Oxford Mathematical Monographs, The Oxford University
  Press, New York, 1998.

\bibitem{MY1982(t)}
William~W. Meeks, III and Shing~Tung Yau, \emph{The classical {P}lateau problem
  and the topology of three-dimensional manifolds. {T}he embedding of the
  solution given by {D}ouglas-{M}orrey and an analytic proof of {D}ehn's
  lemma}, Topology \textbf{21} (1982), no.~4, 409--442.

\bibitem{MY1982(mz)}
\bysame, \emph{The existence of embedded minimal surfaces and the problem of
  uniqueness}, Math. Z. \textbf{179} (1982), no.~2, 151--168.

\bibitem{Mor81}
Hiroshi Mori, \emph{Minimal surfaces of revolution in {$H^{3}$} and their
  global stability}, Indiana Univ. Math. J. \textbf{30} (1981), no.~5,
  787--794.

\bibitem{Seo11}
Keomkyo Seo, \emph{Stable minimal hypersurfaces in the hyperbolic space}, J.
  Korean Math. Soc. \textbf{48} (2011), no.~2, 253--266.

\bibitem{Tuz93}
A.~A. Tuzhilin, \emph{Global properties of minimal surfaces in {${\mathbb R}\sp
  3$} and {${\mathbb H}\sp 3$} and their {M}orse type indices}, Minimal
  surfaces, Adv. Soviet Math., vol.~15, Amer. Math. Soc., Providence, RI, 1993,
  pp.~193--233.

\bibitem{Wang10}
Biao Wang, \emph{Minimal surfaces in quasi-{F}uchsian 3-manifolds}, Math. Ann.
  \textbf{354} (2012), no.~3, 955--966.

\bibitem{Xin03}
Yuanlong Xin, \emph{Minimal submanifolds and related topics}, Nankai Tracts in
  Mathematics, vol.~8, World Scientific Publishing Co. Inc., River Edge, NJ,
  2003.

\end{thebibliography}
%--------------------------------------------------------------------
\end{document}